\documentclass[oneside]{amsart}
\usepackage{etex}
\usepackage{pstricks}
\usepackage{amsmath}
\usepackage{amsfonts}
\usepackage{amsthm}
\usepackage{hyperref}
\usepackage{amscd}
\usepackage{color}
\usepackage{blkarray}
\usepackage{enumerate}
\usepackage{xypic}
\usepackage{tikz}
\usepackage{amsbsy}
\newcommand{\bl}{\boldsymbol{\ell}}
\newcommand{\bx}{\mathbf{x}}
\newcommand{\by}{\mathbf{y}}
\newcommand{\bF}{\mathbf{F}}

\newcommand{\KK}{\mathbb{K}}
\newcommand{\PP}{\mathbb{P}}
\newcommand{\CC}{\mathbb{C}}

\newcommand{\pr}{\mathbf{pr}}

\newcommand{\NN}{\mathbb{N}}

\DeclareMathOperator{\perm}{perm}

\newtheorem{thm}{Theorem}[section]
\newtheorem{prop}[thm]{Proposition}
\newtheorem{lemma}[thm]{Lemma}

\newtheorem{conj}[thm]{Conjecture}

\newtheorem{lem}[thm]{Lemma}
\theoremstyle{definition}

\newtheorem{defn}[thm]{Definition}
\newtheorem{rem}[thm]{Remark}

\newtheorem{ex}[thm]{Example}

\title{Fano Schemes for Generic Sums of Products of Linear Forms}

\author{Nathan Ilten}
\address{Department of Mathematics, Simon Fraser University,
8888 University Drive, Burnaby BC V5A1S6, Canada}
\email{nilten@sfu.ca}

\author{Hendrik S\"u\ss{}}
\address{School of Mathematics, 
The University of Manchester
Alan Turing Building,
Oxford Road,
Manchester M13 9PL,
United Kingdom
}
\email{hendrik.suess@manchester.ac.uk}

\begin{document}

\begin{abstract}
We study the Fano scheme of $k$-planes contained in the hypersurface cut out by a generic sum of products of linear forms. In particular, we show that under certain hypotheses, linear subspaces of sufficiently high dimension must be contained in a coordinate hyperplane. We use our results on these Fano schemes to obtain a lower bound for the product rank of a linear form. This provides a new lower bound for the product ranks of the $6\times 6$ Pfaffian and $4\times 4$ permanent, as well as giving a new proof that the product and tensor ranks of the $3\times 3$ determinant equal five. Based on our results, we formulate several conjectures. 
\end{abstract}
\maketitle
\section{Introduction}
Given an embedded projective variety $X\subset \PP^n$, its Fano scheme $\bF_k(X)$ is the fine moduli space parametrizing projective $k$-planes contained in $X$. Such Fano schemes have been considered extensively for the case of sufficiently general hypersurfaces \cite{altman:77a, barth:81a,langer:97a} but less so for particular hypersurfaces \cite{harris:98a, beheshti:06a, ilten:15a}. In this article, we study the Fano schemes $\bF_k(X)$ for the special family of irreducible hypersurfaces 
\[
X=X_{r,d}=V\left(\sum_{i=1}^r\prod_{j=1}^d c_{ij}x_{ij}\right)\subset \PP^{rd-1}, \quad c_{ij} \in \KK^*
\]
for any $r>1,d>2$. Up to projective equivalence these hypersurfaces do not depend on the choice of the  $c_{ij} \in \KK^*$. Hence, in the following we may assume that $c_{ij}=1$ for all $1 \leq i \leq r$ and  $1 \leq j \leq d$. Moreover, for $d > 2$ the hypersurfaces $X_{r,d}$ are always singular along a union of coordinate hyperplanes of codimension $2r$. We exclude the case $d=2$ since this is a smooth quadric hypersurface with significantly different behaviour.

In \cite[\S 3]{ilten:16a}, Z.~Teitler and the first author considered the Fano scheme $\bF_5(X_{4,3})$. With the help of a computer-assisted calculation, they observed the curious fact that every $5$-plane $L$ of $X_{4,3}$ is either contained in a coordinate hyperplane, or there exist $1\leq a < b \leq 4$ such that $L$ is contained in $V(x_{a1}x_{a2}x_{a3}+x_{b1}x_{b2}x_{b3})$. 
This motivates the following definition:
\begin{defn}[$\lambda$-splitting]
Consider $\lambda\in\NN$.
A $k$-plane $L$ contained in $X_{r,d}$ admits a $\lambda$-\emph{splitting} if there exist $1\leq a_1 < a_2 <\ldots <a_\lambda\leq r$ such that $L$ is contained in 
\[
V\left(\sum_{i=1}^\lambda\prod_{j=1}^d x_{a_ij}\right)\subset \PP^{rd-1}.
\]
We say that $\bF_k(X_{r,d})$ is $m$-split if every $k$-plane of $X_{r,d}$ admits a  $\lambda$-splitting for some $\lambda\leq m$.
\end{defn}
\noindent The above-mentioned observation from \cite{ilten:16a} can now be rephrased as the statement that $\bF_5(X_{4,3})$ is two-split.

We make two conjectures regarding the splitting behaviour of these Fano schemes:

\begin{conj}[One-Splitting]\label{conj:one-split}
Assume $r\geq 2$ and $d\geq 3$. The Fano scheme $\bF_k(X_{r,d})$ is one-split if and only if 
\[
k\geq \begin{cases}
 \frac{r}{2}\cdot d  & r\ \textrm{even}\\
 \frac{r-1}{2}\ \cdot d + 1 & r\ \textrm{odd}\\
\end{cases}.
\]
\end{conj}

\begin{conj}[Two-Splitting]\label{conj:two-split}
Assume $r$ is even and $d\geq 3$. The Fano scheme $\bF_k(X_{r,d})$ is two-split if 
\[
k\geq \frac{r}{2}\cdot d-1.
\]
\end{conj}

\noindent We show in Example \ref{ex:sharp} that the bound on $k$ of Conjecture \ref{conj:one-split} is indeed necessary for one-splitting. However, the sufficiency of the conditions of Conjectures \ref{conj:one-split} and \ref{conj:two-split} for one- and two-splitting is less obvious. Our belief in these conjectures is motivated by our Theorem \ref{thm:conj} below and the connection with the property $C^d_m$ described below. It would be interesting to formulate conjectures characterizing $m$-splitting of $\bF_k(X_{r,d})$ in general (including the case $r$ odd and $m=2$) but we don't know what they would be. For our application below (Theorem \ref{thm:bound}), understanding the cases $m=1$ and $m=2$ suffices.

The following example illustrates the ideas we will use to attack Conjectures \ref{conj:one-split} and \ref{conj:two-split}:
\begin{ex}[$\bF_5(X_{4,3})$ is two-split]
	We are considering the hypersurface $X_{4,3}$ in $\PP^{11}$, equipped with coordinates $x_{11},\ldots,x_{43}$. Any $5$-plane $L$  in $\PP^{11}$ can be represented as the rowspan of a full rank $6\times 12$ matrix $B=(b_{\alpha,ij})$, with rows indexed by $\alpha=0,\ldots ,5$ and column $ij$ corresponding to the homogeneous coordinates $x_{ij}$ on $\PP^{11}$.
We define linear forms $y_{ij}$ in $\KK[z_0,\ldots, z_5]$ by 
\[y_{ij}=\sum_\alpha b_{\alpha,ij}z_\alpha
\]
and note that $L\subset X_{4,3}$ if and only if the form 
\[
	h:=\sum_{i=1}^4\prod_{j=1}^3 y_{ij}
\]
is equal to zero.

Since $B$ has full rank, we can assume that there is a submatrix $B'$ of $B$ consisting of six columns which form the identity matrix. Grouping the columns of $B$ into the $4$ blocks whose indices $ij$ have the same $j$ value, we see that either:
\begin{enumerate}
	\item There are two blocks which each contain at least two columns of $B'$. This implies $h=f_1z_0z_1+f_2z_2z_3-\bl_1-\bl_2$, where $f_1,f_2$ are linear forms and $\bl_1,\bl_2$ are products of three linear forms. Or,
	\item Every block contains at least one column of $B'$, and there is a block containing three columns of $B'$. This implies $h=z_0z_1z_2+f_1z_3+f_2z_4+f_3z_5$.
\end{enumerate}

If $L\subset X_{4,3}$ and hence $h=0$, the second case cannot occur, since $z_0z_1z_2$ is not in the ideal generated by $z_3,z_4,z_5$. In the first case, the equation $h=0$ translates to
\[
\bl_1+\bl_2=z_0z_1f_1+z_2z_3f_2.
\]
We will see in  \S \ref{sec:sums} that this is only possible if either one of $f_1,f_2$ vanishes (in which case $L$ is one-split), or (after permuting indices) $\bl_1=z_0z_1f_1$ (in which case $L$ is two-split).
\end{ex}

More generally, in our study of $\bF_k(X_{r,d})$ we are led to 
consider degree $d>1$ homogeneous equations of the form 
\begin{equation}\label{eqn:sumprod}
\bl_1+\ldots+\bl_m=\sum_{i=0}^mf_i\bx_i
\end{equation}
where the $\bx_i$ are pairwise coprime squarefree monomials, the $\bl_i$ are degree $d$ products of linear forms (possibly equal to $0$, and the $f_i$ are degree $d-\deg \bx_i$ products of linear forms in some polynomial ring (also possibly equal to $0$).
The following property will be essential in our analysis:
\begin{defn}[Property $C^d_m$]
We say that  $C^d_m$ is true if, for any equation of the form \eqref{eqn:sumprod} satisfying $\deg \bx_i+\deg \bx_j\geq d+2$ for all $i\neq j$, it follows that there is some $i$ for which $f_i=0$. 
\end{defn}
\begin{ex}
Property $C_1^2$ is simply stating the obvious fact that for variables $x_1,\ldots,x_4$, the form $x_1x_2+x_3x_4$ is not a product of linear forms. Property $C_1^3$ states the less-obvious fact that for variables $x_1,\ldots,x_5$ and non-zero linear form $f$, the form $x_1x_2x_3+fx_4x_5$ is not a product of linear forms.  
\end{ex}

Our first main result relates the above definition to our two conjectures:
\begin{thm}\label{thm:reduction}
Fix $r\geq 2$, $d\geq 3$ such that either 
\begin{enumerate}
\item $d$ is even,
\item $r$ is even and $d\geq r$, or
\item $r$ is odd and $r\leq 5$.
\end{enumerate}
Suppose that $C^d_m$ is true for all \[m\leq \frac{r-1}{2}.\] Then Conjectures \ref{conj:one-split} and \ref{conj:two-split} hold for this choice of $(r,d)$.
\end{thm}

Secondly, we use this to prove our conjectures in some special cases.
\begin{thm}\label{thm:conj}
Property $C^d_m$ is true if $m\leq 2$ or if $d\leq 4$. Furthermore, Conjectures \ref{conj:one-split} and \ref{conj:two-split} hold if $r\leq 6$  or if $d= 4$.
\end{thm}

Our analysis of Equation \eqref{eqn:sumprod} makes use of relatively elementary methods. However, a more sophisticated approach should also be possible. Equation \eqref{eqn:sumprod} posits that $\sum_{i=0}^mf_i\bx_i$ is a point in the $(m-1)$th secant variety of a Chow variety parametrizing degree $d$ products of linear forms. Equations for the Chow variety are classical, going back to Brill and Gordan \cite{GKZ}. More recently, Y.~Guan has provided some equations for secant varieties of Chow varieties \cite{guan2,guan1} . It would be interesting to see if these equations shed light on the vanishing of the $f_i$ from Equation \eqref{eqn:sumprod}.

Our motivation for studying $\bF_k(X_{r,d})$ is twofold. Firstly, we wish to add to the body of examples of varieties $X$ for which one understands the geometry of $\bF_k(X)$. 
If the Fano scheme $\bF_k(X_{r,d})$ is $m$-split for some $m<r$, then $k$-dimensional linear subspaces of $X_{r,d}$ can be understood in terms of linear subspaces of $X_{r',d}$ for certain $r'<d$.
We illustrate this by describing the irreducible components of $\bF_k(X_{r,d})$ for $k\geq(r-2)(d-1)+1$ whenever  $r\leq d+1$ or $d=4$, see Examples \ref{ex:one} and \ref{ex:three}.
We also characterize when $\bF_k(X_{r,d})$ is connected, see Theorem \ref{thm:connected}.

Secondly, we may use our results to obtain lower bounds on the \emph{product rank} of certain linear forms. Recall that the \emph{product rank} (also known as Chow rank) of a degree $d$ form $f$ is the smallest number $r$ such that we can write 
\begin{equation*}
f=\bl_1+\ldots+\bl_r
\end{equation*}
where the $\bl_i$ are products of $d$ linear forms. We denote the product rank of $f$ by $\pr(f)$. Note that product rank may be used to give a lower bound on \emph{tensor rank}, see \cite[\S1.3]{ilten:16a} for details.

A form $f$ in $n+1$ variables is \emph{concise} if it cannot be written as a form in fewer variables after a linear change of coordinates.
The hypersurface $V(f)$ of a concise degree $d$ form of product rank at most $r$ is isomorphic to the intersection of $X_{r,d}\subset \PP^{rd-1}$ with an $n$-dimensional linear subspace. This allows us to relate properties concerning linear subspaces contained in $V(f)$ to product rank. 
Generalizing \cite[Theorem 3.1]{ilten:16a}, we prove the following:
\begin{thm}\label{thm:bound}
Let $f$ be a concise irreducible degree $d>1$ form in $n+1$ variables such that $V(f)\subset \PP^{n}$ is covered by $k$-planes, and let $r\in \NN$.
\begin{enumerate}
\item If  $\bF_k(X_{r,d})$ is one-split, then $\pr(f)\neq r$.
\item If $r$ is even, $k>n-r$, and \[\bF_k(X_{r,d}),\ \bF_{k-d}(X_{r-2,d}),\ldots,\ \bF_{k-\frac{d(r-4)}{2}}(X_{4,d})\] are two-split, then $\pr(f)\neq r$.   
\end{enumerate}
\end{thm}

Applying this to the $3\times 3$ determinant of a generic matrix, we recover that its product and tensor ranks are five \cite{ilten:16a}. Note  that we have replaced the computer-aided computation of $\bF_{5}(X_{3,4})$ with a conceptual proof. 
We may also apply Theorem \ref{thm:bound} to the $4\times 4$ determinant $\det_4$ of a generic matrix to obtain $\pr(\det_4)\geq 7$; this is equal to the lower bound one obtains from Derksen and Teitler's lower bound on the Waring rank \cite{derksen:15a}. In Example \ref{ex:pfaffian} we apply the theorem to the Pfaffian $f$ of a generic $6 \times 6$ skew-symmetric matrix to obtain $\pr(f)\geq 7$, beating the previous lower bound of $6$.
Finally, in Example \ref{ex:perm} we use a slightly different argument to obtain that the product rank (and tensor rank) of the $4\times 4$ permanent is at least $6$, beating the previous lower bound of $5$.

The rest of the paper is organized as follows. In \S \ref{sec:sums}, we study equations of the form \eqref{eqn:sumprod}. We use this in \S \ref{sec:fano} to show our splitting results for the Fano schemes $\bF_{k}(X_{r,d})$, as well as studying several cases in more detail. Finally, we prove Theorem \ref{thm:bound} in \S \ref{sec:pr}  and apply our results to a number of examples including the $6\times 6$ Pfaffian and $4\times 4$ permanent.

For simplicity, we will be working over an arbitrary algebraically closed field $\KK$. Note however that all our main results clearly hold for arbitrary fields simply by restricting from $\KK$ to any subfield.  

\section{Special Sums of Products of Linear Forms}\label{sec:sums}
\subsection{Preliminaries}
In this section we will prove that property $C^d_m$ holds for $m \leq 2$ or $d \leq 4$ . We will obtain this result by using induction arguments. These arguments involve a refined version of the $C^d_m$ property. Consider an equation of the form
  \begin{equation}
    \bl_1+\ldots+\bl_m=\sum_{i=1}^{k+n} f_i\bx_i \label{eq:2},
  \end{equation}
  with $n > m$. As before, the $\bx_i$ are pairwise coprime squarefree monomials and  the $\bl_i$ are degree $d$ products of linear forms, possibly equal to zero.
  In fact, whenever we say that some polynomial $g$ is homogeneous of degree $d$, we include the possibility that $g=0$.

  We now assume simply that $f_i$ are degree $(d-\deg \bx_i)$ forms (or zero), no longer requiring that they be products of linear forms. It will be convenient to order the summands on the right hand side so that \[\deg \bx_1 \leq \deg \bx_2 \leq \ldots \leq \deg \bx_{k+n}.\]
We will maintain this ordering convention throughout all of \S\ref{sec:sums}.
Similar to the property $C^d_m$ we make the following definition.

  \begin{defn}[Property $C^d_{k,m,n}$]
    We say that $C^d_{k,m,n}$ is true, if for any equation of the form \eqref{eq:2}
    satisfying
    \begin{align}
      \deg \bx_i+\deg \bx_j\geq d+1 && \text{ for } i > k\label{eq:3}\\
      \deg \bx_i+\deg \bx_j\geq d+2 && \text{ for } i,j > k.\label{eq:4}
    \end{align}
    it follows that there are $i_1, \ldots, i_{n-m} > k$ for which $f_{i_j}=0$. 
  \end{defn}
\noindent  Note that by definition $C^d_{0,m,m+1}$ implies $C_m^d$.

  \begin{lem}
    \label{lem:multiple-vanishing}
    Fix $d$ and $m$ and assume that $C^d_{k,m,m+1}$ holds for every $k \geq 0$. Then
    $C^d_{k,m,n}$ holds for every $n > m$.
  \end{lem}
  \begin{proof}   
We may argue by induction on $n$. Obviously, the hypotheses for $C^d_{k,m,n}$ imply  those for $C^d_{k+1,m,n-1}$. Hence, by the induction hypothesis we have the vanishing of $(n-m-1)$ of the $f_i$, with $i> k+1$. Using this we end up with the hypotheses for $C^d_{k,m,n-((n-m)-1)}=C^d_{k,m,m+1}$ being fulfilled, and we see that another of the $f_i$ with $i > k$ has to vanish. 
\end{proof}

\subsection{Property $C_1^d$}
We will first analyze the case $m=1$. Therefore, we consider an equation of the form
\begin{equation}\label{eqn:form}
\bl=\sum_{i=1}^r f_i\bx_i
\end{equation}
with $f_i$ forms of degree $(d-\deg \bx_i)$ (possibly equal to zero), and $\bl$ a product of linear forms, also possibly equal to zero. As before, we order indices such that $\deg \bx_1\leq \deg \bx_2\leq \ldots\leq \deg \bx_r$.

\begin{rem}[Cancellation]\label{rem:cancellation}
Assume we are given a variable $x$ which divides $\bl$, one monomial $\bx_i$, and all $f_j$ for $j\neq i$. Setting
\begin{align*}
\bl'={\bl}/{x}\qquad
\bx_j'=\begin{cases}
{\bx_j}/{x}&i=j\\
\bx_j &i\neq j
\end{cases}
\qquad f_j'=\begin{cases}
{f_j}/{x}&i\neq j\\
f_j &i= j
\end{cases}
\end{align*}
leads to 
\begin{equation*}
\bl'=\sum_{i=1}^r f_i'\bx_i'
\end{equation*}
where we have reduced from forms of degree $d$ to degree $d-1$.
We call this the \emph{cancellation} of \eqref{eqn:form} by $x$.
\end{rem}

\begin{lemma}\label{lemma:dcoeff}
Let $l$ be a linear form dividing $\sum f_i\bx_i$, where the $f_i$ are forms of degree $(d-\deg \bx_i)$. 
\begin{enumerate}
\item If $\deg \bx_1+ \deg \bx_2\geq d+2$, then for all $i$, $l$ divides $\bx_i$ or $f_i$.
\item If $\deg \bx_1 +\deg \bx_2 \geq d+1$ and $l$ is a monomial, then for all $i$, $l$ divides $\bx_i$ or $f_i$.
\end{enumerate}
\end{lemma}
\begin{proof}
We first prove the second statement. We have
\[\sum f_i\bx_i=xg\]
for some variable $x$ and form $g$. Expanding the left hand side as a sum of monomials, we see that the degree condition ensures that no terms from $f_i\bx_i$ cancel with $f_j\bx_j$ for $i\neq j$. But every monomial on the right hand side is divisible by $x$, hence also on the left hand side. The claim follows.

For the first statement, we reduce to the second by performing a change of coordinates taking $l$ to a monomial. This can be achieved while preserving all variables in the $\bx_1,\ldots,\bx_r$ with at most one exception, say in $\bx_i$. After factoring out this one linear form from $\bx_i$, the pairwise of sum degrees is still at least $d+1$ and we may apply the second claim.
 \end{proof}

\begin{lemma}\label{lemma:divide}
Suppose $\bl=\sum_{i=1}^r f_i\bx_i$ for $r\geq 2$, $f_i$ forms of degree $(d- \deg \bx_i)$. Assume $\bl\neq 0$, and let $\lambda$ be the number of distinct factors of $\bl$. If
\[
\left\lceil \frac{\prod \deg \bx_i}{\lambda}\right\rceil > \frac{\prod \deg \bx_i}{\deg \bx_1 \cdot \deg \bx_2}
\]
then there is a variable $x$ dividing both $\bl$ and one of the $\bx_i$. This is true if $\deg \bx_1 \cdot \deg \bx_2 >\lambda$, in particular if $\deg \bx_1 \cdot \deg \bx_2>d$.
\end{lemma}
\begin{proof}
For each $\bx_i$, choose some variable $x_i$ dividing it. Setting $x_1=x_2=\ldots=x_r=0$ will result in the equality $\bl=0$, hence there exists one factor of $\bl$ depending only on $x_1,\ldots,x_r$. There are $\prod_i \deg \bx_i$ possible ways to choose the $x_i$, and $\lambda$ factors of $\bl$, so there must be one factor of $\bl$  which depends only on the $x_1,\ldots,x_r$ for 
 \[
\left\lceil \frac{\prod \deg \bx_i}{\lambda}\right\rceil 
\]
different choices. On the other hand, the intersection of more than 
\[
\frac{\prod \deg \bx_i}{\deg \bx_1 \cdot \deg \bx_2}
\]
choices of the $x_1,\ldots,x_r$ contains at most one variable. Hence, if the above inequality is satisfied, the claim follows.
\end{proof}

\begin{rem}
If $\bl=0$, the conclusion of the above lemma is trivial since every variable divides $\bl$.
\end{rem}

\begin{prop}\label{prop:rkonebasis}
Suppose
\[
\bl=f_1\bx_1+f_2\bx_2
\]
with $f_i$ forms of degree $(d-\deg \bx_i)$.
\begin{enumerate}
\item If $\deg \bx_1 + \deg \bx_2 \geq d+2$, then either $f_1$ or $f_2$ vanishes.\label{part:a1}
\item If $\bl$ is not squarefree and $\deg \bx_1 + \deg \bx_2 \geq d+1$, then either $f_1$ or $f_2$ vanishes.\label{part:a2}
\end{enumerate}

In particular, property $C^d_1$ holds for every $d>0$.
\end{prop}
\begin{proof}
For the first case, the hypothesis  $\deg \bx_1 + \deg \bx_2\geq d+2$ implies in particular that $\deg \bx_1 \geq 2$ and $\deg \bx_1 \cdot  \deg \bx_2 >d$. Hence, 
 Lemma \ref{lemma:divide} implies the existence of a variable $x$ dividing both $\bl$ and one of the $\bx_i$. But then $f_j\bx_j$ is divisible by $x$ for $j\neq i$, hence $x$ divides $f_j$. Cancelling by $x$, we may proceed by induction on the degree $d$.

For the second case, we proceed with a similar argument. The inequality \[\deg \bx_1 + \deg \bx_2 \geq d+1\]implies that $\deg \bx_1  \cdot \deg \bx_2>d-1$, which is larger than or equal to the number of distinct factors of $\bl$. Thus, we again find a variable $x$ dividing both $\bl$ and one of the $\bx_i$. If in fact $x^2$ divides $\bl$, then after factoring out one power of $x$ from $\bx_i$ and $f_j$, Lemma \ref{lemma:dcoeff} guarantees that $x$ divides $f_i$ as well. Dividing $\bl$, $f_i$, and $f_j$ by $x$, we reduce to the first case.

If $x^2$ does not divide $\bl$, we may cancel by $x$ as in the first case, maintaining that $\bl/x$ is still not squarefree. To finish, we again proceed by induction on the degree $d$.
\end{proof}
\begin{rem}\label{rem:rkone}
It is clear that the degree bounds in Proposition \ref{prop:rkonebasis} cannot be improved upon. If $\deg \bx_1 + \deg \bx_2= d+1$ and $x_1,x_2$ are variables dividing $\bx_1,\bx_2$ respectively, then setting $f_i=\bx_i/x_i$ gives
\[
f_1\bx_1+f_2\bx_2=(x_1+x_2)\frac {\bx_1}{x_1}\frac{\bx_2}{x_2}.
\]
Likewise, if $\deg \bx_1+ \deg \bx_2\leq d$ there are non-trivial degree $d$ syzygies between $\bx_1$ and $\bx_2$, so we cannot expect the second claim to hold.
\end{rem}

We next prove a stronger version of $C_{k,1,2}^d$.

\begin{prop}\label{prop:rkone}
Suppose for $r\geq 3$, $d \geq 2$
\begin{equation}\label{eqn:form2}
\bl=\sum_{i=1}^r f_i\bx_i
\end{equation}
with $f_i$ being forms of degree $(d-\deg \bx_i)$. Then if $\deg \bx_1 + \deg \bx_{r-1}\geq d+1$, some $f_{i}$ with $\deg \bx_i\geq \deg \bx_{r-1}$ must vanish.

In particular, $C^d_{k,1,2}$ holds for every $d > 0$, $k \geq 0$.
\end{prop}
\begin{proof}
Set $\alpha=\deg \bx_1$ and $\beta=\deg \bx_{r-1}$. It suffices to prove the proposition in the case that $\deg \bx_1=\ldots=\deg \bx_{r-2}=\alpha$ and $\deg \bx_{r-1}= \deg \bx_r=\beta$. Indeed, we may absorb variables from $\bx_2,\ldots,\bx_{r-2},\bx_r$ into the corresponding $f_i$ to reduce to this case. Henceforth we will assume we are in such a situation.

We begin by proving the claim when $r=3$ and $\alpha=1$, that is, $\bx_1$ is a single variable $x$. By using Proposition \ref{prop:rkonebasis}(\ref{part:a1}), we see that modulo $x$, either $f_2$ or $f_3$ must vanish. But since $\alpha+\beta\geq d+1$, $f_2$ and $f_3$ are both just constants, hence one must vanish outright.

Next, we consider the case when $r=3$ and $\alpha>1$. First, we show that some $f_i$ must vanish, with no restriction on its degree. We apply Lemma \ref{lemma:divide} to find a variable $x$ dividing some monomial $\bx_i$ and $\ell$. Applying Lemma \ref{lemma:dcoeff}, we may conclude that $x$ divides $f_j$ for $j\neq i$. In particular, if $d=2$, this implies that $f_j=0$ for $j\neq i$. For $d>2$, we may cancel by $x$ to reduce the degree by one and conclude by induction on degree that $f_i=0$ for some $i$.
Now we show that we can impose the desired degree restriction on $f_i$. Indeed, if $i=2,3$, or $i=1$ and $\alpha=\beta$, this is automatic. If instead $i=1$ and $\alpha<\beta$, then we have 
$\bl=f_2\bx_2+f_3\bx_3$ satisfying the hypotheses of Proposition \ref{prop:rkonebasis}(\ref{part:a1}), from which the claim follows.

It remains to consider the cases when $r\geq 4$. We will now induct on $r$. First assume that $\alpha<\beta$. By setting any variable $x$ in $\bx_i$ equal to zero for $i\leq r-2$, we reduce to an equation of the form \eqref{eqn:form2} with one fewer summand on the right hand side, yet $\alpha$, $\beta$, and $d$ the same. Hence, by induction, $x$ divides $f_{r-1}$ or $f_{r}$. Now, there are $\alpha\cdot (r-2)$ variables appearing in the $\bx_i$ for $i\leq r-2$, yet 
\[\deg f_{r-1}+\deg f_r=2(d-\beta)\leq 2(\alpha-1).\]
 Thus, if $r>3$ then either $f_{r-1}$ or $f_r$ must vanish.

If instead $r\geq 4$ and $\alpha=\beta$, we may again apply Lemma \ref{lemma:divide} followed by Lemma \ref{lemma:dcoeff} to find a variable $x$ dividing some $\bx_i$ and $f_j$ for $j\neq i$. We may reorder the monomials such that $i=1$, since all have the same degree $\alpha$.
Cancelling by $x$, we again find ourselves in the situation of \eqref{eqn:form2}, but now with $\alpha<\beta$, so by the above, $xf_{r-1}$ or $xf_{r}$ vanishes, thus $f_{r-1}$ or $f_r$ does as well.  
\end{proof}
\begin{rem}
Proposition \ref{prop:rkone} is sharp in the following sense. Suppose that in \eqref{eqn:form2}, we have $\deg \bx_1 + \deg \bx_{r-1} \leq d$. Then \emph{none} of the $f_i$ need vanish. Indeed, for $i=2,\ldots,r-1$ we can take $f_i=g_i\bx_1$ for any forms $g_i$ of degree $(d-\deg \bx_1 - \deg \bx_i)$, and 
\[f_1=\sum_{i=2}^{r-1}-g_i\bx_i.\]
Then 
\[\sum_{i=1}^r f_i\bx_i=f_r\bx_r\]
so if $f_r$ is a product of linear forms, then so is the whole sum, yet for appropriate choice of $g_i$ none of the $f_i$ will vanish. 
\end{rem}

\subsection{Property $C_2^d$}
We now move to the case of $C_2^d$:

\begin{prop}\label{prop:rktwobasis}
Suppose 
\begin{equation}\label{eqn:rktwobasis}
\bl_1+\bl_2=f_1\bx_1+f_2\bx_2+f_3\bx_3
\end{equation}
with $f_i$ forms of degree $(d-\deg \bx_i)$. If $\deg \bx_1+\deg \bx_2\geq d+2$, then either $f_1$, $f_2$, or $f_3$ vanishes.

In particular, $C^d_2$ holds for every $d > 0$.
\end{prop}
\begin{proof}
We will prove the statement by induction on the degree $d$. For $d=1$ there is nothing to prove. If we can show that $\bl_1,\bl_2$ have a common factor $l$, then we are done. Indeed, by Lemma \ref{lemma:dcoeff}, $l$ must divide either each $\bx_i$ or $f_i$. Pulling $l$ out of each $f_i$ where we can, and out of $\bx_i$ in at most one position, allows us to ``cancel by $l$'' in a fashion similar to Remark \ref{rem:cancellation}. We thus reduce the degree and the claim follows by induction.

In the following, we will assume that no common factor $l$ of $\bl_1$ and $\bl_2$ exists, and that all $f_i$ are non-zero. 
For simplicity, we may assume that $\deg \bx_2= \deg \bx_3$, since this case implies the more general one. We denote $\deg \bx_1$ by $\alpha$, and $\deg \bx_2$ by $\beta$. Our hypothesis on degrees is now simply $\alpha+\beta\geq d+2$.

Consider any factor $l$ of $\bl_1$ or $\bl_2$. Setting $l=0$, we reduce to the case of Proposition \ref{prop:rkonebasis}(\ref{part:a1}) (if $l$ divides some $\bx_i$) or Proposition \ref{prop:rkone} (by absorbing into some $f_i$ a variable of $\bx_i$ appearing in $l$). In either case, we see that modulo $l$, some $f_i$ must vanish, that is, $l$ is a factor of $f_i$. We may proceed to do this for all distinct divisors of $\bl_1$ and $\bl_2$. But since \[\deg f_1+\deg f_2+\deg f_3 \leq 2(d-\beta)+(d-\alpha),\] we conclude that together $\bl_1$ and $\bl_2$ have at most 
\[
2(d-\beta)+(d-\alpha)\leq 2d-\beta-2
\]
distinct factors.
It follows that either both $\bl_1$ and $\bl_2$ contain a square, or else that the non-squarefree product has at most $d-\beta-2$ distinct factors.

Assume first that $\bl_1$ is squarefree, and fix some factor $l$. We now argue in a similar fashion to the proof of Lemma \ref{lemma:divide}.
For each $\bx_1,\bx_2,\bx_3$, fix a variable $y_i$ so that $l$ and all remaining variables are linearly independent. 
For each $\bx_i$, choose some variable $x_i\neq y_i$ dividing $\bx_i$. Setting $x_1=x_2=x_3=0$ will result in the equality $\bl_1=-\bl_2$, hence one factor of $\bl_2$ is zero modulo $x_1,\ldots,x_m,l$. 
 There are $(\alpha-1)(\beta-1)^2$ possible ways to choose the $x_i$, and at most $d-\beta-2$ factors of $\bl_2$, so there must be one fixed factor of $\bl_2$  which is zero mod $x_1,\ldots,x_m,l$ for
 \[
\left\lceil \frac{(\alpha-1)(\beta-1)^2}{d-\beta-2}\right\rceil 
\]
different choices. On the other hand, the intersection of more than $\beta^2$ 
choices of the $x_1,\ldots,x_m$ contains no variable. Hence, since $d-\beta-2<\alpha-1$ it follows that there is a factor of $\bl_2$ which is zero modulo $l$, that is, agrees with it.

We now instead assume that both $\bl_1$ and $\bl_2$ contain  factors with multiplicity at least two. 
Consider any factor $l$ of $\bl_1$ or $\bl_2$. As long as $l$ is not a variable in $\bx_2$ or $\bx_3$, we may set $l=0$ and  conclude that $l$ divides $f_2$ or $f_3$. Indeed, if $l$ divides  $\bx_1$ this follows from Proposition \ref{prop:rkonebasis}. Otherwise we may absorb into some $f_i$ a variable of $\bx_i$ made linear dependent modulo $l$, and then apply Proposition \ref{prop:rkone} followed by Proposition \ref{prop:rkonebasis}(\ref{part:a2}) to conclude that two of $f_1$, $f_2$, and $f_3$ vanish modulo $l$.

If at most one factor $l$ of $\bl_1$, $\bl_2$ divides $\bx_2$ or $\bx_3$ but not $f_2$ or $f_3$, we thus obtain that $\bl_1,\bl_2$ have at most $1+2(d-\beta)$ distinct factors. But then either $\bl_1$ or $\bl_2$ has at most $d-\beta$ distinct factors, so an argument similar to the previous case where $\bl_1$ squarefree above shows that $\bl_1$ and $\bl_2$ would have to possess a common factor.

So we now finally consider the case that at least two distinct factors $x,y$ of $\bl_1,\bl_2$ are variables found in $\bx_2$ and $\bx_3$, neither dividing $f_2$ or $f_3$. 
It follows by Proposition \ref{prop:rkonebasis} that each such factor must divide $f_1$. Without loss of generality, we assume that $x$ divides $\bx_3$ and $\bl_1$. 
We obtain
\begin{align*}
\bl_2\equiv f_2\bx_2 \mod x,
\end{align*}
so $\bl_2$ has $\beta$ factors which only depend on $x$ and a single variable of $\bx_2$.

If $y$ divides $\bx_3$ and $\bl_2$, then setting $x=y=0$, we obtain $f_2\bx_2\equiv 0$ $\mod x,y$, a contradiction. 
If instead $y$ divides $\bx_3$ and $\bl_1$
we obtain
\begin{align*}
\bl_2\equiv f_2\bx_2 \mod y
\end{align*}
and $\bl_2$ has $\beta$ factors which only depend on $y$ and a single variable of $\bx_2$.
Since $\bl_2$ has $d<2\beta$ factors, one must also just be a variable $w$ of $\bx_2$. So in this case, we conclude that a variable $w$ of $\bx_2$ divides $\bl_2$.
If instead $y$ divides $\bx_2$, we see by setting $x=0$ that $y$ must divide $\bl_2$, so we can take $w=y$ to produce $w$ as above. 

We thus may assume that we are in the situation of variables $x,w$ with $x$ dividing $\bx_3$ and $\bl_1$, and $w$ dividing $\bx_2$. By Proposition \ref{prop:rkonebasis}, $w$ divides $f_j$ for $j=1$ or $j=3$. 
Now let $k\in\{1,3\}$ be such that $i\neq j$. 
We thus obtain
\begin{align*}
\bl_1\equiv f_k\bx_k \mod w,
\end{align*}
hence $\bl_1$ has $\deg \bx_k $ factors which depend on $w$ and a single variable of $\bx_k$.

The right hand side of Equation \eqref{eqn:rktwobasis} clearly contains monomials divisible by $\bx_j$. But the left hand side cannot: while each monomial of $\bl_1$ has degree at least $\deg \bx_k$ in the variables of $\bx_k$ and $w$, and each monomial of $\bl_2$ has degree at least $\beta=\deg \bx_2$ in the variables of $\bx_2$ and $x$, the part of $\bx_j$ relatively prime to $x$ has degree at least $\deg \bx_j-1$. The inequality
\[
\deg \bx_j-1+ \deg \bx_k\geq d+1
\]
then shows that this impossible. We conclude that in fact some $f_i$ must equal zero. 
\end{proof}

\begin{rem}
Proposition \ref{prop:rktwobasis} is optimal. Indeed, suppose that $\deg \bx_1 + \deg \bx_2 \leq d+1$. Then by Remark \ref{rem:rkone}, for appropriate non-vanishing choices of $f_1,f_2$, $f_1\bx_1+f_x\bx_2$ is a product of linear forms, so $f_1\bx_1+f_2\bx_2+f_3\bx_3$ is a sum of two products of linear forms for any choice of $f_3$.
\end{rem}

\subsection{Property $C_m^d$ for $d\leq 4$}
We now prove a lemma that will help us with the degree four case:

\begin{lemma}\label{lemma:help}
Fix $d\geq 2$, $k>0$, and $m>0$ and let $n=m+1$. 
Assume that (\ref{eq:3}) and (\ref{eq:4}) hold and that $C_{k',m',n'}^d$ holds whenever $m'<m$, or whenever $m'=m$ and $k'<k$.
In Equation \eqref{eq:2}, consider any linear form $l$ dividing $p$ of the summands $\bl_i$ on the left hand side. Then $l$ must also divide $p$ of the $f_i$ with $i>k$.
\end{lemma}
\begin{proof}
Assume $l$ divides some factor $x_j$ of $\bx_j$. Setting $l=0$, we now have that the hypothesis for $C^d_{k,m-p,m}$ is fulfilled. Since we have assumed that $C^d_{k,m-p,m}$ is true, $p$ of the $f_i$ with $i > k$ must vanish modulo $l$.

Even if $l$ does not divide any $\bx_j$, we still may set $l=0$, modifying the right hand side of the equation $\bl_1+\ldots+\bl_m=\sum_{i=1}^{k+n} f_i\bx_i$
to replace one factor of some $\bx_j$ by a linear form $f$ which is no longer a monomial. Now we have to distinguish two cases. Let us assume first that $j > k$. Then, since the degree of $\bx_j$ drops by one, we are in the situation of $C^d_{k+1,m-p,m}$. As before by our assumption,  $p$ of the $f_i$ with $i > k$ must vanish modulo $l$. 

If $j \leq k$ then the fact that the degree of $\bx_j$ drops may violated condition (\ref{eq:4}). However, we may bring the summand $f_j \bx_j$ to the left hand side of the equation. This leaves us in the situation of $C^d_{k-1,m-p+1,m+1}$ and our assumption again provides the vanishing of $p$ of the $f_i$, with $i > k$.
\end{proof}

We now use this lemma to show $C_m^d$ for arbitrary $m$ and $d\leq 4$:

\begin{prop}\label{prop:degfour}
If $d\leq 4$, property $C_{k,m,n}^d$ holds for arbitrary $k>0$ and $n>m>0$. In particular, $C_m^d$ holds for $m>0$.
\end{prop}
\begin{proof}
  By Lemma~\ref{lem:multiple-vanishing} it is enough to show that $C^d_{k,m,m+1}$ holds for all $m,k>0$. Now, we prove $C^d_{k,m,m+1}$ by induction on $m$ and $k$.
Note that, for $k$ arbitrary, $C^d_{k,1,2}$ follows from Proposition~\ref{prop:rkone}. Moreover, Lemma~\ref{lem:multiple-vanishing} then provides  $C^d_{k,1,n}$ for arbitrary $n > 1$.
  
Assume we have proven property $C^d_{k',m',m'+1}$ is true whenever $m' < m$, or whenever $m'=m$ and  $k' < k$.  For $d \leq 4$ we have $\sum_{i>k} \deg f_i \leq m+1$. 
Either one of the $f_i$ has to vanish, which would prove our claim, or by Lemma \ref{lemma:help}, all the linear factors of the $\bl_i$ occur as one of the  (at most) $(m+1)$ linear factors of the $f_i$ for $i>k$. Here, by a \emph{linear factor} we we mean an equivalence class of linear forms, where two linear forms are equivalent if one is a non-zero scalar multiple of the other.

By the above the linear factors of the $f_i$ form a multiset $L$ of cardinality at most $m+1$ and every $\bl_j$  is divisible by one of the elements of $L$. On the other hand, we have seen by Lemma \ref{lemma:help} that every $l\in L$ may divide at most $m(l)$ of the $\bl_j$, where $m(l)$ denotes the multiplicity of $l$ in $L$. Since $\# L \leq (m+1)$, there can be at most one $\bl_j$ which is divisible by more than one linear factor. This implies, we have $\bl_i = l_i^d$ for all but one of the summands on the left hand side. 

For $m > 2$ this implies that we can write the left hand side as the sum of only $(m-1)$ products, since we can write $l_i^d + l_j^d$ as a product of linear forms. Then using the induction hypothesis for $C^d_{k,m-1,m}$ concludes the proof for the case $m > 2$. 

To conclude, we consider the case $m=2$. If none of the $f_i$ (with $i > k$) vanishes, we have seen that at most three linear factors $l_1$, $l_2$ and $l_3$ can occur on the left hand side. We choose $\lambda \in \KK$ and set $l_2 = \lambda l_3$. Now, the left hand side depends only on two linear forms. In this situation the left hand side is actually a product of linear forms (since $\KK$ is algebraically closed). In the same way as in the proof of Lemma \ref{lemma:help} (when setting one of the $l_i$ to zero) we see by the induction hypothesis that one of the $f_i$ has to be divisible by $(l_2 - \lambda l_3)$. We have only finitely many choices for the linear factors of the $f_i$, but we have infinitely many choices for $\lambda \in \KK$. Hence, there are $\lambda, \lambda' \in \KK$ with $(l_2 - \lambda' l_3)$ dividing  $(l_2 - \lambda l_3)$, which implies either $l_2 = 0$ or $l_3=0$. In any case, one of the summands $\bl_1$ or $\bl_2$ has to vanish, and we are in the case of $C^d_{k,1,3}$.
\end{proof}

\subsection{Consequences}
The following lemma derives a consequence of the property $C^d_m$ which will be used later.

\begin{lemma}\label{lemma:uniqueness}
Consider an equation of the form 
\begin{equation}\label{eqn:unique}
\bl_1+\ldots+\bl_m=\sum_{i=1}^m\bx_i
\end{equation}
where the $\bl_i$ are degree $d\geq 3$ products of linear forms, and the $\bx_i$ are pairwise relatively prime squarefree monomials of degree $d$. If property $C_{m-1}^d$ is true, then there is a permutation $\sigma\in S_m$ such that $\bl_i=\bx_{\sigma(i)}$ for all $i$.
\end{lemma}
\begin{proof}
Consider any factor  $l_i$ of some $\bl_i$. If $l_i$ does not divide any $\bx_j$, we may set $l_i=0$, modifying the right hand side of Equation \eqref{eqn:unique} to replace one factor of some $\bx_j$ by a linear form $f$ which is no longer a monomial. But this equation still satisfies the hypotheses necessary for $C_{m-1}^d$, as long as $d\geq 3$, so in fact, $l_i$ must have divided one of the $\bx_j$ all along.

We thus see that every factor of each $\bl_i$ is just a variable, up to scaling. By comparing the monomials on both sides of \eqref{eqn:unique}, we find the desired permutation.
\end{proof}

\begin{rem}
We may interpret the above lemma geometrically as saying that, if $C_d^{r-1}$ is true, then the subgroup of $PGL(rd-1)$ taking $X_{r,d}$ to itself is generated by the semidirect product of the torus 
\[
T=\{x_{11}\cdots x_{1d}=x_{21}\cdots x_{2d}=\ldots=x_{r1}\cdots x_{rd}\}
\]
with the copy of the symmetric group $S_r$ permuting the indices $i$ of $x_{ij}$, and the $r$ copies of $S_d$ permuting the indices $j$ of $x_{ij}$ for some fixed $1\leq i \leq r$.
\end{rem}

\section{Fano Schemes and Splitting}\label{sec:fano}
\subsection{Main results}\label{sec:mainproof}
In this section, we will prove Theorems \ref{thm:reduction} and \ref{thm:conj}.
For $n=rd-1$, consider projective space $\PP^{n}$ with coordinates $x_{ij}$, $1\leq i \leq r$ and $1\leq j \leq d$.
Let $L$ be a $k$-dimensional linear subspace of $\PP^{n}$. We may represent $L$ as the rowspan of a full rank $(k+1)\times rd$ matrix $B=(b_{\alpha,ij})$, with rows indexed by $\alpha=0,\ldots ,k$ and column $ij$ corresponding to the homogeneous coordinates $x_{ij}$ on $\PP^{n}$.
We define linear forms $y_{ij}$ in $S=\KK[z_0,\ldots, z_k]$ by 
\[y_{ij}=\sum_\alpha b_{\alpha,ij}z_\alpha,
\]
along with degree $d$ forms 
\[
\by_i=\prod_{j=1}^d y_{ij}.
\]
The condition that $L$ is contained in $X_{r,d}$ is equivalent to the condition 
\begin{equation}\label{eqn:infano}
\sum_i \by_{i}=0.
\end{equation}
The condition that $L$ is one-split is equivalent to the condition that some
$y_{ij}$ vanishes, and also to  the condition that some $\by_i$ vanishes.  The condition that $L$ is two-split is equivalent to the existence of $a \leq a_1 < a_2 \leq r$ such that $\by_{a_1} + \by_{a_2}=0$.
\begin{ex}[$k$-planes which are not one-split]\label{ex:sharp}
For $r=2m$ and $k=md-1$, let $L$ be any $k$-plane with $y_{ij}$ all linearly independent for $i\leq m$, $y_{(i+m)1}=-y_{i1}$, and $y_{(i+m)j}=y_{ij}$ for $j>1$. Then clearly $L$ is contained in $X_{r,d}$, but is not one-split (although it is two-split).

For $r=2m+1$ and $k=md$, consider forms $y_{ij}$ satisfying $\{y_{ij}\}_{i\leq m}$ and $y_{r1}$ all linearly independent, and
\begin{align*}
y_{(i+m)1}&=-y_{i1}\qquad&\textrm{for}\ i<m,\\
y_{(i+m)j}&=y_{ij}\qquad&\textrm{for}\ i<m,\ j>1,\\
y_{rj}&=y_{(2m)j}=y_{mj}&\textrm{for}\ j>1,\\
y_{(2m)1}&=-y_{m1}-y_{r1}.
\end{align*}
Let $L$ be the corresponding $md$-plane. Clearly $L$ is contained in $X_{r,d}$, but is not one-split. 

We thus see that the bound on $k$ in Conjecture \ref{conj:one-split} is sharp.
\end{ex}

We henceforth assume that $L\subset X_{r,d}$, that is, that $\sum_i \by_i=0$, and that none of the $\by_i$ vanish, that is, $L$ is not one-split. Without loss of generality, we may inductively reorder the forms $\by_i$ as follows: given $\by_1,\ldots,\by_s$, we take $\by_{s+1}$ to be any form such that the dimension of the vector space spanned by the $\{y_{ij}\}_{i\leq s+1}$ is maximal. 

After this re-ordering, we may define integers $\lambda_1,\lambda_2,\ldots,\lambda_r$ inductively by requiring that the dimension of the vector space  spanned by the $\{y_{ij}\}_{i\leq s}$ is equal to $\sum_{i \leq s} \lambda_i$.
By the way we have ordered the forms $\by_i$, this implies that $\lambda_1\geq \lambda_2\geq \ldots \geq \lambda_r$.
Indeed, by our choice of $\by_s$ we must have $\sum_{i \leq s} \lambda_i \geq \lambda_{s+1}+ \sum_{i \leq s-1} \lambda_i$ and, hence, $\lambda_s \geq \lambda_{s+1}$ for every $1 \leq s \leq r$.

We may then choose a new basis 
\[
z_{ij},\qquad 1\leq i \leq r\qquad 1\leq j \leq \lambda_i;
\]
for the degree one piece of $S$ with the property that each $z_{ij}$ is a factor of $\by_i$, and each factor of $\by_i$ is in the span of 
\[\{z_{hj}\}_{1 \leq h \leq i,1\leq j \leq   \lambda_h}.\]

We now will assume that 
\[
k\geq \begin{cases}
 \frac{r}{2}\cdot d-1  & r\ \textrm{even}\\
 \frac{r-1}{2}\ \cdot d + 1 & r\ \textrm{odd}\\
\end{cases}.
\]

 \begin{lemma}\label{lemma:help2}
  For $s\geq 0$, suppose that $\lambda_{r-s}=0$. If $r$ is odd, then $s \leq \frac{r-3}{2}$ . If $r$ is even, then $s \leq \frac{r}{2}-1$
  In particular, we always have $s+1 \leq r-s-1$.
 \end{lemma}
 \begin{proof}
   We have
   \begin{equation}
 k+1=\sum_{i=1}^r \lambda_i=\sum_{i=1}^{r-s-1}\lambda_i \leq (r-s-1)d. \label{eq:help}
 \end{equation}
For $r$ odd, our assumptions on $k$ imply
   $\frac{r-1}{2}d + 2\leq (r-s-1)d$. Hence, we have $s+\frac{2}{d} \leq \frac{r-1}{2}$ which implies $s < \frac{r-1}{2}$. Since $s$ is an integer we obtain $s \leq \frac{r-3}{2}$.

For $r$ even, \eqref{eq:help} implies  $\frac{r}{2}d \leq (r-s-1)d$, which directly implies the claim. 
 \end{proof}

\begin{lemma}\label{lemma:lambdas}
For $s\geq 0$, suppose that $\lambda_{r-s}=0$. 
Assume further that either
\begin{enumerate}
\item $d$ is even,
\item $r$ is even and $d\geq r-2s$, or
\item $r$ is odd and $r-2s\leq 6$.
\end{enumerate}
Then $\lambda_{s+1}+\lambda_{s+2}\geq d+2$. 
\end{lemma}
\begin{proof}
We have that 
\[
k+1=\sum_{i=1}^r \lambda_i=\sum_{i=1}^{r-s-1} \lambda_i \leq sd+\sum_{i={s+1}}^{r-s-1} \lambda_i.
\]
Note that by Lemma~\ref{lemma:help2} we have  $s+1 \leq r-s-1$, so the summation on the right hand side makes sense.
Using our assumption on $k$  we thus have 
\begin{equation}\label{eq:ineq}
\sum_{i={s+1}}^{r-s-1} \lambda_i
\geq
\begin{cases}
 \frac{r-2s}{2}\cdot d  & r\ \textrm{even}\\
 \frac{r-2s-1}{2}\ \cdot d + 2 & r\ \textrm{odd}\\
\end{cases}.
\end{equation}
Suppose that $\lambda_{s+1}+\lambda_{s+2}\leq d+1$.
If $d$ is even, then $\lambda_{s+2}\leq \frac{d}{2}$, and thus $\lambda_{i}\leq \frac{d}{2}$ for all $i\geq s+2$. But then
\[
\sum_{i={s+1}}^{r-s-1} \lambda_i\leq
(d+1)+(r-2s-3)\frac{d}{2}=\frac{r-2s-1}{2}\cdot d+1
\]
contradicting \eqref{eq:ineq}, since $d\geq 3$.

Assume instead that $d$ is odd. Then  $\lambda_{i}\leq \frac{d+1}{2}$ for all $i\geq s+2$, so 
\[
\sum_{i={s+1}}^{r-s-1} \lambda_i\leq
(r-2s-1)\frac{d+1}{2}.
\]
But this contradicts \eqref{eq:ineq} if $r$ is even and $d\geq r-2s$, or if $r$ is odd and $r-2s\leq 6$.
\end{proof}

\begin{proof}[Proof of Theorem \ref{thm:reduction}]
First note that $\lambda_r=0$. Indeed, if not, then $\by_r$ contains a factor which is not in the span of the factors of the $\by_i$ for $i<r$, so it is impossible to satisfy Equation \eqref{eqn:infano}.
Suppose that we have inductively shown that $\lambda_{r-s}=0$ for some $s\leq \frac{r-1}{2}-1$. Then by Lemma \ref{lemma:lambdas}, we have that $\lambda_{s+1}+\lambda_{s+2}\geq d+2$. If $\lambda_{r-s-1}\neq 0$, we set $z_{i1}=0$ for $i=s+3,\ldots,r-s-1$ and use property $C_{s+1}^d$ applied to 
\[
-\sum_{i=r-s}^r\by_i =\sum_{i=1}^{s+2} \by_i \qquad\mod \{z_{i1}\}_{s+3\leq i\leq r-s-1}
\]
to conclude that some $\by_i$ for $i\leq s+2$ vanishes modulo  $\{z_{i1}\}_{s+3\leq i\leq r-s-1}$. But by our construction of the $\by_i$, this is impossible, and we conclude that $\lambda_{r-s-1}=0$.

We proceed in this fashion until we obtain $\lambda_{t}=0$ for
\[
t=\left\lceil\frac{r}{2}\right\rceil+1,
\]
since for 
\[
s=\left\lfloor \frac{r-1}{2}-1\right\rfloor,
\]
we have
\[
t\geq r-s-1.
\]

If $r$ is odd, we conclude again by Lemma \ref{lemma:lambdas} that $\lambda_{t-2}+\lambda_{t-1}\geq d+2$, and an appropriate application of property $C_{(r-1)/2}^d$ shows that some $\by_i$ must vanish, a contradiction.
If $r$ is even, we must have $\lambda_1=\ldots=\lambda_{r/2}=d$. This is impossible if $k$ satisfies the bound of Conjecture \ref{conj:one-split}, completing the claim regarding one-splitting. For the claim regarding two-splitting, we may apply Lemma \ref{lemma:uniqueness} to conclude that  $\by_1=-\by_j$ for some $j>r/2$. But this implies two-splitting.
\end{proof}

\begin{proof}[Proof of Theorem \ref{thm:conj}]
The first part of the Theorem is simply Propositions \ref{prop:rkonebasis}, \ref{prop:rktwobasis}, and \ref{prop:degfour}. The statement regarding Conjectures \ref{conj:one-split} and \ref{conj:two-split} following immediately from Theorem \ref{thm:reduction} except in the cases ($r=4$, $d=3$), ($r=6$, $d=3$). and ($r=6$, $d=5$). 
The obstruction in all these cases comes about that in the proof of Theorem \ref{thm:reduction}, we cannot use Lemma \ref{lemma:lambdas} to conclude that $\lambda_1+\lambda_2\geq d+2$. However, we may use Proposition \ref{prop:rkone} to compensate. 

Consider for example the case $r=6$, $d=3$. 
If $\lambda_1+\lambda_2\leq d+1=4$, then we must in fact have $\lambda_1=\lambda_2=\ldots=\lambda_4=2$, and $\lambda_5=1$. Setting $z_{51}=0$, we may apply Proposition \ref{prop:rkone} to reach a contradiction. Thus, $\lambda_5=\lambda_6=0$. A similar argument shows that $\lambda_3=0$ as well, and we conclude as in the proof of Theorem \ref{thm:reduction}.
The other two cases are similar, and left to the reader.
\end{proof}

\begin{rem}\label{rem:3d}
Consider the Fano scheme $\bF_k(X_{r,d})$. If we only know that $C^d_m$ is true for all $m\leq M$ for some $M$ strictly less than $(r-1)/2$, we may still use the above arguments to conclude that $\bF_k(X_{r,d})$ is one-split if $k$ is sufficiently large.

For example, we know that $C^d_m$ is always true for $m=1,2$. For $r\leq 6$, we already know by Theorem \ref{thm:conj} exactly when $\bF_k(X_{r,d})$ is one-split, so assume that $r\geq 7$. We claim that if $k\geq d(r-3)$, then $\bF_k(X_{r,d})$ must be one-split.
Indeed, if $d$ is even, Lemma \ref{lemma:lambdas} applies and arguing as in the proof of Theorem \ref{thm:conj} shows that if $L$ is not one-split, then $\lambda_r=\lambda_{r-1}=\lambda_{r-2}=0$. But this contradicts $k\geq d(r-3)$. For $d$ odd, slightly more care is needed. Assume some $k$-plane $L$ is not one-split.
The arguments from Lemma \ref{lemma:lambdas} will apply if \[k\geq \frac{(r-1)(d+1)}{2},\] in which case we are done as above. But this inequality is satisfied except for the case $r=7$, $d=3$. As always, $\lambda_7=0$.
 But certainly $\lambda_2+\lambda_3\geq d+1=4$, so using Proposition \ref{prop:rkone} in place of $C^3_1$ we conclude that $\lambda_{6}=0$. But then one easily verifies that $\lambda_2+\lambda_3\geq d+2=5$, so $\lambda_5=0$, which is impossible.
\end{rem}

\subsection{Consequences and examples}
We now want to use our results on splitting to study the geometry of $\bF_k(X_{r,d})$. We first note the following result:

\begin{thm}\label{thm:connected}
The Fano scheme $\bF_k(X_{r,d})$ is non-empty if and only if $k< r(d-1)$. Such a Fano scheme is connected if and only if $k< r(d-1)-1$.
\end{thm}
\begin{proof}
Consider the subtorus $T$ of $(\KK^*)^{rd}$ cut out by 
\[
x_{11}x_{12}\cdots x_{1d}=x_{21}x_{22}\cdots x_{2d}=\ldots=x_{r1}x_{r1}\cdots x_{rd}.
\]
This torus acts naturally on $\PP^{rd-1}$. Since it fixes $X_{r,d}$, this action induces an action on $X_{r,d}$, and hence also on $\bF_k(X_{r,d})$. It is straightforward to check that the only $k$-planes of $\PP^{rd-1}$ fixed by $T$ are intersections of coordinate hyperplanes. Thus, any torus fixed point of $\bF_k(X_{r,d})$ corresponds to a $k$-plane $L$ whose associated non-zero forms $y_{ij}$  of \S \ref{sec:mainproof} are all linearly independent. 

Recall that such a $k$-plane $L$ is contained in $\bF_k(X_{r,d})$ if and only if Equation \eqref{eqn:infano} is satisfied. But since by assumption the non-zero $y_{ij}$ are linearly independent, this is equivalent to requiring that for each $i$, there is some $j$ such that $y_{ij}=0$. Since every component of $\bF_k(X_{r,d})$ must contain a torus fixed point, it follows immediately that if $k\geq r(d-1)$, $\bF_k(X_{r,d})$ must be empty. The non-emptiness of $\bF_k(X_{r,d})$ for $k< r(d-1)$ is also clear.

Assume now that $k=r(d-1)-1$. By Remark \ref{rem:3d}, it inductively follows that any $k$-plane of $X_{r,d}$ must be torus fixed. But there are $r^d$ such fixed $k$-planes, so $\bF_k(X_{r,d})$ is not connected.

Suppose finally that $k<r(d-1)-1$, and let $L$ be a torus fixed $k$-plane contained in $\bF_k(X_{r,d})$. We prove that $\bF_k(X_{r,d})$ is connected by deforming $L$ to a $k$-plane satisfying \[y_{11}=y_{21}=\ldots=y_{r1}=0.\] Since the set of all such $k$-planes forms a connected subscheme of $\bF_k(X_{r,d})$ isomorphic to the Grassmannian $G(k+1,r(d-1))$, and every irreducible component of the Fano scheme contains a torus fixed point, it follows that $\bF_k(X_{r,d})$ is connected. 

To see that we can deform $L$ to a $k$-plane of the desired type, let $j_1,\ldots,j_r$ be such that $y_{1j_1},\ldots,y_{rj_r}$ all vanish; these must exist since $L$ is torus fixed and contained in $X_{r,d}$. Let $i$ be the smallest index for which $j_i\neq 1$. The set of all $k$-planes satisfying 
$y_{1j_1}=\ldots=y_{rj_r}=0$ forms a closed subscheme of $\bF_k(X_{r,d})$ isomorphic to $G(k+1,r(d-1))$. Since $k<r(d-1)-1$, this set contains a $k$-plane $L'$ satisfying  $y_{i1}=0$ along with $y_{1j_1}=\ldots=y_{rj_r}=0$, and $L$ deforms to $L'$. Replacing $L$ with $L'$ we can continue this procedure until we arrive at a $k$-plane satisfying $y_{11}=y_{21}=\ldots=y_{r1}=0$ as desired.
\end{proof}

\begin{rem}
	Theorem \ref{thm:connected} is in stark contrast to the situation for the Fano scheme $\bF_k(X)$ for a general degree $d>2$ hypersurface $X\subset \PP^n$. Such a Fano scheme $\bF_k(X)$ is non-empty if and only if
	$\phi(n,k,d)\geq 0$, and connected if $\phi(n,k,d)\geq 1$, where
\[
	\phi(n,k,d)=(k+1)(n-k)-{{k+d}\choose{k}},
\]
see \cite{langer:97a}.
\end{rem}

We now illustrate on several examples how our results help determine the irreducible component structure of $\bF_k(X_{r,d})$.
\begin{ex}[$\bF_k(X_{r,d})$ for $k\geq (r-2)(d-1)+2$]\label{ex:one}
For $k\geq (r-2)(d-1)+2$ and $r\geq 3$, Conjecture \ref{conj:one-split} would imply that $\bF_k(X_{r,d})$ is one-split. Assume this to be true. Considering any $k$-plane $L$ contained in $X_{r,d}$, we know that some $x_{ij}$ must vanish. Intersecting $L$ with $x_{i1}=x_{i2}=\ldots=x_{id}=0$, we obtain a linear subspace $L'$ in $X_{r-1,d}$ of dimension $k'$, where $k'\geq k-(d-1)\geq (r-3)(d-1)+2$. Hence, $L'$ is also (conjecturally) one-split, as long as $r-1\geq 3$. We may proceed in this fashion until we obtain a linear subspace $L''$ in $X_{2,d}$ of dimension $k''$, where $k''\geq k-(r-2)(d-1)\geq 2$. If $L''$ is also one-split, then $L$ is contained in an $(r(d-1)-1)$-plane of the form 
\[
x_{1j_1}=x_{2j_2}=\ldots=x_{rj_r}=0
\]
for some choice of $j_1,\ldots,j_r$. The $k$-planes in this fixed $(r(d-1)-1)$-plane are parametrized by the Grassmannian $G(k+1,r(d-1))$. This leads to $d^r$ irreducible components of $\bF_k(X_{r,d})$, each isomorphic in its reduced structure to $G(k+1,r(d-1))$.

If on the other hand $L''$ is not one-split, then Equation \eqref{eqn:infano} implies that after some permutation in the $j$ indices, $y_{1j}$ and $y_{2j}$ are linearly dependent for all $j$. In particular, $L''$ is contained in a $(d-1)$-plane of $X_{2,d}$, appearing in a $d-1$-dimensional family. Thus, the plane $L$ is contained in an $(r-1)(d-1)$-plane of $X_{r,d}$, which is moving in a $(d-1)$-dimensional family.
This only can occur if $k\leq (r-1)(d-1)$. In such cases, it follows that the 
corresponding irreducible component of $\bF_k(X_{r,d})$ has dimension $(d-1)+(k+1)((r-1)(d-1)-k)$, and there are 
\[
{r\choose 2} d^{r-2}\cdot (d!)
\]
such components.

To summarize, the Fano scheme has two types of irreducible components:
\begin{itemize}
\item {\bf Type A}: $d^r$ components of dimension $(k+1)(r(d-1)-(k+1))$, isomorphic in their reduced structures to a Grassmannian; general $k$-planes in such components are  contained in the intersection of $r$ coordinate hyperplanes.
\item  {\bf Type B}:  Assuming $k\leq (r-1)(d-1)$, ${r\choose 2} d^{r-2}\cdot (d!)$ components of dimension $(d-1)+(k+1)((r-1)(d-1)-k)$; general $k$-planes in such components are contained in the intersection of $(r-2)$ coordinate hyperplanes. 
\end{itemize}
This analysis relied on Conjecture \ref{conj:one-split}. By Theorem \ref{thm:conj}, this holds true if $r \leq 6$ or $d=4$, so we know our above conclusions are true as long as this is satisfied. Furthermore, by Remark \ref{rem:3d}, the one-splitting we need follows if $k\geq d(r-3)$. But this is always satisfied as long as $r\leq d+2$.
\end{ex}

The above example is somewhat elementary, since all the Fano schemes appearing in the reduction steps are one-split or two-split. However,  if we understand the structure of a Fano scheme which isn't one-split (or even two-split), we can leverage this to an understanding of $\bF_k(X_{r,d})$ for larger values of $r$. We will illustrate in the next two examples.

\begin{ex}[Special components of $\bF_d (X_{3,d})$]\label{ex:base}
By Example \ref{ex:sharp}, we know that $\bF_d(X_{3,d})$ is not one-split; since $r=3$, it is also not two-split. Nonetheless, with a bit of work, we can completely describe these Fano schemes. In this example, we will describe a special type of irreducible component; all components will be dealt with in Example \ref{ex:three}.

 We begin with the case $d=2$ (although we usually have been assuming $d>2$).
The variety $X_{3,2}$ is just a non-singular quadric fourfold; it is well-known that $\bF_2(X_{3,4})$ is the disjoint union of two copies of $\PP^3$.
%Let $L$ be any $2$-plane of $X_{3,2}$ which is not one-split. After re-ordering indices, we may assume that $y_{11}, y_{12}$, and $y_{21}$ are linearly independent. Setting $y_{11}=y_{21}=0$, we see that either $y_{31}$ or $y_{32}$ depends only on $y_{11},y_{21}$. A similar statement holds when setting $y_{12}=y_{21}=0$. We conclude that, up to permutation of $y_{31}$ and $y_{32}$, 
%\begin{align*}
%y_{31}=\alpha y_{11}+\alpha a y_{21}\\
%y_{32}=\beta y_{12}+\beta b y_{21}
%\end{align*}
%since if e.g. $y_{31}$ depended only on $y_{21}$, then $y_{21}$ must divide $y_{11}y_{12}$, a contradiction.
%Now using Equation \eqref{eqn:infano}, we conclude that
%\[
%y_{22}=-by_{11}-ay_{12}-aby_{21}
%\]
%and that $\alpha\beta=-1$. For any choice of $i\neq j$ and $l$, $y_{i1},y_{i2},y_{jl}$ are linearly independent, so our initial assumption was unnecessary. We thus see that there are exactly two components of $\bF_2(X_{3,2})$ not consisting of only one-split $2$-planes, corresponding to our above choice of permutation of $y_{31}$ and $y_{32}$. Each component has dimension three. Utilizing the natural action of $(\KK^*)^3$ and considering the weights of the non-vanishing Pl\"ucker coordinates, we see that these components are both isomorphic to the three-dimensional projective toric variety corresponding to the configuration of lattice points
%\begin{align*}
%\pm e_1,\quad\pm e_2,\quad\pm e_3\\
%-e_1+e_2+e_3,\quad e_1-e_2+e_3,\quad e_1+e_2-e_3\\
%-e_1-e_2-e_3
%\end{align*}
%in $\ZZ^3$. It is straightforward to check that this corresponds to $\PP^3$, embedded via $\CO(2)$.

We now suppose that $d>2$. Let $L$ be a $d$-plane of $X_{3,d}$ which is not one-split. After reordering indices, we may assume that $y_{11},\ldots,y_{1\lambda_1},y_{21},\ldots,y_{2\lambda_2}$ are linearly independent, with $\lambda_1+\lambda_2\geq d+1$ and $\lambda_1\geq \lambda_2$. If $\lambda_2=1$, then we may replace $\lambda_1$ with $\lambda_1-1$ and $\lambda_2$ with $\lambda_2+1$, unless all $y_{2j}$ are linearly dependent. But this is easily seen to contradict Equation \eqref{eqn:infano}. So we may assume that $\lambda_2\geq 2$. 

For each choice $y_{1j_1},y_{2j_2}$ with $j_i\leq \lambda_i$, by setting
$y_{1j_1}=y_{2j_2}=0$, we find that one $y_{3j}$ depends only on  $y_{1j_1},y_{2j_2}$. A simple counting argument shows that some $y_{3j}$ can only depend on some $y_{1j_1}$ or $y_{2j_2}$. But by Equation \eqref{eqn:infano}, this form must also divide some $y_{2j_2'}$ or respectively $y_{1j_1'}$ (for $j_i'>\lambda_i$). Factoring this out of Equation \eqref{eqn:infano}, we arrive at the situation of a $k'$-plane in $X_{3,d-1}$, with $k'\geq d-1$. If $k'>d-1$, then this plane is one-split, contradicting our assumption, so in fact $k'=d-1$. We continue in this fashion of reducing degree until we arrive at one of the two toric components of $\bF_2(X_{3,2})$. The component of $\bF_d(X_{3,d})$ is a $(d-2)$-fold iterated $\PP^2$ bundle over the toric component, and hence has dimension $3+2(d-2)=2d-1$. There are 
\[
2\cdot {d \choose{2}}^3 \left({(d -2)}!\right)^2
\]
such components:
we choose one of two toric components of $\bF_2(X_{3,2})$; then for each index $i$ we choose two of the $y_{ij}$ which are not getting factored out. We then match each of the remaining $y_{1j}$ with a $y_{2j'}$ and $y_{3j''}$, of which there are 
$\left((d -2)!\right)^2$ ways.

\end{ex}

We now leverage the above example to lower the bound on $k$ in Example \ref{ex:one} by one:

\begin{ex}[$\bF_{k}(X_{r,d})$ for $k=(r-2)(d-1)+1$]\label{ex:three}
Let $L$ be any $k$-plane of $X_{r,d}$, for $k=(r-2)(d-1)+1$.
Similar to in Example \ref{ex:one}, Conjecture \ref{conj:one-split} would imply that $L$ is one-split, as long as $r\geq 5$. As in Example \ref{ex:one}, we successively reduce to a $k'$-plane in $X_{4,d}$, with $k'\geq 2(d-1)+1=2d-1$. By Theorem \ref{thm:conj}, $L'$ is two-split.

Suppose first that $L'$ is not one-split. Then after permuting $\{1,2,3,4\}$, we may assume that $\by_1+\by_2=\by_3+\by_4=0$. 
The factors of $\by_1$ and $\by_2$ must agree up to scaling, and similarly for $\by_3$ and $\by_4$. Similar to the component of type $B$ in Example \ref{ex:one}, we see that $L'$ is a $2d-1$-plane of $X_{4,d}$, moving in a $2(d-1)$-dimensional family. Thus, the plane $L$ also is moving in a $2(d-1)$-dimensional family. It follows that the corresponding irreducible component of $\bF_k(X_{r,d})$ has dimension $2(d-1)$,
 and there are 
\[
{r\choose r-4,2,2} d^{r-4}\cdot (d!)^2
\]
such components.

If $L'$ is one-split, we may reduce further to a $k''$-plane $L''$ in $X_{3,d}$ with $k''\geq d$. Suppose next that $L''$ is not one-split. Then $k''=d$, and  $L''$ corresponds to a point in one of the $(2d-1)$-dimensional irreducible components described in Example \ref{ex:base}. 
 It follows that the corresponding irreducible component of $\bF_k(X_{r,d})$ has dimension $2d-1$,
 and there are 
\[
{r \choose 3}d^{r-3}\cdot 2\cdot {d \choose{2}}^3 \left({(d -{2})}!\right)^2
\]
such components.

Finally, if $L''$ is also one-split, then we get components of types A and B similar to those appearing in Example \ref{ex:one}. To summarize, assuming that the necessary splitting conjectures are true, $\bF_{k}(X_{r,d})$ for  $k=(r-2)(d-1)+1$ has the following irreducible components:
\begin{itemize}
\item {\bf Type A}: $d^r$ components of dimension \[2((r-2)(d-1)+2)(d-2),\] isomorphic in their reduced structures to a Grassmannian; general $k$-planes in such components are  contained in the intersection of $r$ coordinate hyperplanes.
\item  {\bf Type B}:  ${r\choose 2} d^{r-2}\cdot (d!)$ components of dimension \[(d-1)+( (r-2)(d-1)+2)(d-2);\] general $k$-planes in such components are contained in the intersection of $(r-2)$ coordinate hyperplanes. 
\item  {\bf Type C}:  there are \[{r \choose 3}d^{r-3}\cdot 2\cdot {d \choose{2}}^3 \left({(d -{2})}!\right)^2\] components of dimension $2d-1$; general $k$-planes in such components are contained in the intersection of $(r-3)$ coordinate hyperplanes. 
\item  {\bf Type D}:  there are \[{r\choose r-4,2,2} d^{r-4}\cdot (d!)^2\] components of dimension $2(d-1)$; general $k$-planes in such components are contained in the intersection of $(r-4)$ coordinate hyperplanes. 
\end{itemize}
This analysis relied on appropriate splitting statements, which (similar to Example \ref{ex:one}) hold true if $r \leq 6$, $d=4$, or $r\leq d+1$.
\end{ex}

\begin{ex}[$\bF_5(X_{4,3})$]
For a concrete example, consider $\bF_5(X_{4,3})$. By Example \ref{ex:three}, we see that this Fano scheme has the following components:

\begin{center}
\begin{tabular}{l r r}
& Dimension & Number\\
\hline
{\bf Type A} & 12 & 81\\
{\bf Type B} & 8 & 324\\
{\bf Type C} & 5 & 648\\
{\bf Type D} & 4 & 216\\
\end{tabular}.
\end{center}
\end{ex}

\section{Product Rank}\label{sec:pr}
\subsection{Bounding product rank}
\begin{proof}[Proof of Theorem \ref{thm:bound}]
Assume that $\pr(f)\leq r$. Since $f$ is concise, this implies that there is an $n$-dimensional linear space $Y\subset \PP^{rd-1}$ such that $V(f)=X_{r,d}\cap Y$. 
Since we are assuming that $V(f)$ is covered by $k$-planes, there must be a positive-dimensional irreducible subvariety $S\subset \bF_k(X_{r,d})$ such that the $k$-planes corresponding to points in $S$ are all contained in $Y$, and that the linear span of these $k$-planes is exactly $Y$.

Now, if all $k$-planes parametrized by $S$ are contained in a coordinate hyperplane of $\PP^{rd-1}$, we can clearly write $f$ as a sum of $r-1$ products of linear forms, that is, $\pr(f)\neq r$. But this is certainly the case if $\bF_k(X_{r,d})$ is one-split.

Assume instead that $r$ is even and the two-splitting assumption of the theorem is fulfilled. As above, if every $k$-plane parametrized $S$ is contained in a coordinated hyperplane, we are done. Otherwise, by the two-splitting assumption, we can permute the indices $i=1,\ldots,r$ such that 
every $k$-plane $L$ parametrized by $S$ is contained in 
\[
V(x_{i1}\cdots x_{id}+x_{(i-1)1}\cdots x_{(i-1)d})
\] 
for $i=2,4,\ldots,r$. Using the notation from \S \ref{sec:fano}, this tells us that 
\begin{equation}\label{eqn:ys}
y_{i1}\cdots y_{id}+y_{(i-1)1}\cdots y_{(i-1)d}=0.
\end{equation}
for $i=2,4,\ldots,r$.
After reordering the $y_{ij}$ for each fixed $i$, we conclude (by unique factorization of polynomials) that the forms $y_{ij}$ and $y_{(i-1)j}$ are proportional for $i=2,4,\ldots,r$ and $j\leq d$ for all $k$-planes $L$ in $S$.

For some fixed $s=2,4,6,\ldots,r$, suppose that the ratio  $y_{sj}/y_{(s-1)j}$ is some constant $c_j$ as $L$ ranges over $S$. Note that these constants satisfy $\prod c_j=-1$.
 Then every $L$ in $S$ is contained in the linear space
\[
V\left(\{x_{sj}-c_jx_{(s-1)j}\}_{j\leq d}\right)
\]
so their span $Y$ is as well. This means that after restricting to $Y$ we have
\[
\sum_{i=1}^r x_{i1}\cdots x_{id}=\sum_{i\neq s,s-1} x_{i1}\cdots x_{id}
\]
so the product rank of $f$ is at most $r-2$.

We have thus arrived in the situation where for each fixed $i=2,4,\ldots,r$, there is some $j\leq d$ such that the ratio between $y_{ij}$ and $y_{(i-1)j}$ is non-constant over $S$. A straightforward calculation shows that the dimension of the span of two general $k$-planes $L,L'$ in $S$ must be at least $k+r$, leading to the inequality $k+r\leq n$; by assumption, this is a contradiction.
\end{proof}

\begin{rem}\label{rem:better}
Suppose that in the situation of part two of Theorem \ref{thm:bound}, we know that the family of $k$-planes $S\subset \bF_k(V(f))$ covering $V(f)$ is $m$-dimensional. Then the hypothesis $k>n-r$ may be replaced with the condition $k>n-m-r/2$. Indeed, in the conclusion of the proof of the theorem, the assumption on the dimension of $S$ guarantees that at least $m$ of the ratios $y_{ij}/y_{(i-1)j}$ vary independently of each other. Combining Equation \eqref{eqn:ys} with the fact that at least one ratio $y_{ij}/y_{(i-1)j}$  varies for each $i$ guarantees that in fact a total of at least $m+r/2$ ratios vary. As above, this shows that the dimension of the span of two general $k$-planes $L,L'$ in $S$ must be at least $k+m+r/2$, leading to the desired contradiction.
\end{rem}

\subsection{Examples of bounds on product rank}

\begin{ex}[$3\times 3$ determinant]\label{ex:det3}
In \cite{ilten:16a}, Z.~Teitler and the first author prove that $\pr(\det_3)>4$ over $\CC$, where $\det_3$ is the determinant of a generic $3\times 3$ matrix. H.~Derksen gave an expression for $\det_3$ as a sum of $5$ multihomogeneous products of linear forms in \cite{derksen:16a}, so we conclude $\pr(\det_3)=5$.
This also shows that the tensor rank of $\det_3$ equals five.

The proof that $\pr(\det_3)$ consisted of a computer calculation showing that $\bF_5(X_{4,3})$ is $2$-split, and then a special case of Theorem \ref{thm:bound}. Our Theorem \ref{thm:conj} makes this computer calculation unnecessary, and is valid in arbitary characteristic. We conclude that the product and tensor ranks of $\det_3$ are at least five over \emph{any} field.

Derksen's identity requires $2$ to be invertible, so we conclude that except in characteristic $2$, $\pr(\det_3)=5$. We do not know whether $\pr(\det_3)$ is $5$ or $6$ in characteristic $2$.
\end{ex}

\begin{ex}[$4 \times 4$ determinant]\label{ex:det4}
Let $\det_4$ be the $4\times 4$ determinant; this is easily seen to be a concise form. The projective hypersurface $V(\det_4)$ is covered by $11$-dimensional linear spaces, see e.g.~\cite{ilten:15a}.
By Theorem \ref{thm:conj}, we know that $\bF_{11}(X_{6,4})$ and $\bF_{7}(X_{4,4})$ are both $2$-split, so we may apply Theorem \ref{thm:bound} to conclude that $\pr(\det_4)\neq 6$. A similar application of Theorem \ref{thm:bound} shows that $\pr(\det_4)\neq 4,5$. If $\pr(\det_4)\leq 3$, then the projective hypersurface $V(\det_4)\subset \PP^{15}$ must be a cone, in which case every maximal linear subspace would contain a common line. But this is not the case, so we conclude that $\pr(\det_4)\geq 7$ (in arbitrary characteristic).

This is exactly the bound on product rank in characteristic zero which follows from Z.~Teitler and H.~Derksen's bound on Waring rank. They show that the Waring rank of $\det_4$ is at least $50$ \cite{derksen:15a}, from which follows that $\pr(\det_4)\geq 7$ by \cite[\S 1.2]{ilten:16a}.

Our above argument for the product rank of $\det_4$ can be generalized to show that, for $n\geq 3$, $\pr(\det_n)\geq 2n-1$, as long as we assume that Conjecture \ref{conj:two-split} holds. However, for $n\geq 5$ this is much worse than the bound that follows from known lower bounds on Waring rank \cite{derksen:15a}.   
\end{ex}

\begin{ex}[$6 \times 6$ Pfaffian]\label{ex:pfaffian}
Let $f$ be the Pfaffian of a generic $6\times 6$ skew-symmetric matrix; this is also a concise form. Derksen and Teitler show that the Waring rank of $f$ is at least $24$ \cite{derksen:15a}. Section 1.2 of \cite{ilten:16a} then implies that $\pr(f)\geq 6$. 

We will use Theorem \ref{thm:bound} to show that $\pr(f)\neq 6$, and hence $\pr(f)\geq 7$, a new lower bound. First note that by Theorem \ref{thm:conj}, $\bF_9(X_{6,3})$ is one-split. Secondly, we have that $V(f)\subset \PP^{14}$ is covered by projective $9$-planes. Indeed, for any $6\times 6$ singular skew-symmetric matrix $A$ with $0\neq v\in\KK^6$ in its kernel, consider the linear space of all $6\times 6$ skew-symmetric matrices $B$ satisfying
\[
B\cdot v=0.
\]
This is clearly a linear space of singular skew-symmetric matrices containing $A$. There are six linear conditions cutting out this linear space, but they are linearly dependent, since
\[
v^\mathrm{tr}\cdot A\cdot v=0.
\]
Hence, $A$ is contained in a linear space of dimension $14-5=9$.
The claim $\pr(f)\neq 6$ now follows from Theorem \ref{thm:bound}.
\end{ex}

For our final example, we must use a different argument than Theorem \ref{thm:bound}, since the permanental hypersurface is not covered by high-dimensional linear spaces:
\begin{ex}[$4\times 4$ permanent]\label{ex:perm}
	Let $\perm_4$ be the permanent of a generic $4\times 4$ matrix. Assume that the characteristic of $\KK$ is not two, in which case Example \ref{ex:det4} applies.

Shafiei has shown that the Waring rank of $\perm_4$ is at least $35$ \cite{shafiei:15a}, from which follows that  $\pr(\perm_4)\geq 5$ by \cite[\S 1.2]{ilten:16a}. We will show that in fact, $\pr(\perm_4)\geq 6$. On the other hand, Glynn's formula gives $8$ as an upper bound for the product rank of $\perm_4$ \cite{glynn:10a}.
Our result also gives a lower bound of $6$ on the \emph{tensor rank} of $\perm_4$.

To fix notation, suppose that $\perm_4$ is the permanent of the matrix 
\[
M=\left(\begin{array}{c c c c}
z_{11} &z_{12}& z_{13}& z_{14}\\
z_{21} &z_{22}& z_{23}& z_{24}\\
z_{31} &z_{32}& z_{33}& z_{34}\\
z_{41} &z_{42}& z_{43}& z_{44}\\
\end{array}
\right).
\]
The hypersurface $V(\perm_4)\subset \PP^{15}$ contains exactly $8$ $11$-planes \cite{ilten:15a}: $H_i$ and $V_j$ for $i,j\leq 4$ are given respectively by the vanishing of the $i$th row or $j$th column of $M$. Any two of the $H_i$, or any two of the $V_j$ span the entire space $\PP^{15}$.

If $\pr(\perm_4)\leq 5$, then $V(\perm_4)$ is isomorphic to $X_{5,4}$ intersected with a $15$-dimensional linear space.
If under the embedding of $V(\perm_4)$ in $X_{5,4}$, two of the $H_i$ or two of the $V_j$ are contained in a common coordinate hyperplane, it follows that $\pr(\perm_4)\leq 4$, contradicting the above bound. Thus, we may assume that this is not the case.

On the other hand, it follows from Theorem \ref{thm:conj} that any $11$-plane of $X_{5,4}$ is contained in the intersection of three coordinate hyperplanes. Since no pair of $H_i$ or $V_j$ is contained in a common coordinate hyperplane, some pair $(H_a,V_b)$ must be contained in a common coordinate hyperplane. Now, $H_a$ and $V_b$ span the $14$-dimensional linear space $L=V(z_{ab})\subset \PP^{15}$. But since this $14$-dimensional linear space is contained in a coordinate hyperplane of $X_{5,4}$, we conclude that $\pr(\perm_4')\leq 4$, where $\perm_4'$ is obtained from $\perm_4$ by setting $z_{ab}=0$. We will show that this cannot be. 

Indeed, in such a situation we would have $V(\perm_4')\subset \PP^{14}$ isomorphic to the intersection of $X_{4,4}$ with a $14$-dimensional linear space. Intersecting $H_i$ and $V_j$ with $L$, we arrive at a set of $8$ $11$- or $10$-dimensional planes $H_i',V_j'$ with properties similar to above. Similar to above, if a pair of $H_i'$ or $V_j'$ is contained in a coordinate hyperplane in $X_{4,4}$, then we would have that $\pr(\perm_4')\leq 3$. But if this is not the case, an argument similar to above shows that $\pr(\perm_4'')\leq 3$, where  $\perm_4''$ is obtained from $\perm_4$ by setting two variables equal to zero. The key step of the argument is here another application of Theorem \ref{thm:conj} showing that any $10$-plane of $X_{4,4}$ is contained in four coordinate hyperplanes.

To arrive at the final contradiction, first note that  $\pr(\perm_4')\leq 3$ implies $\pr(\perm_4'')\leq 3$. The latter implies in particular that one can write $\perm_4''$ as a form in $12$ variables, that is $V(\perm_4'')\subset\PP^{13}$ must be a cone. However, utilizing the natural torus action on $V(\perm_4'')$ similar to in \cite[Proposition 2.3]{ilten:15a}, one easily verifies that the intersections of $H_1,H_2,H_3,H_4$ with this $\PP^{13}\subset \PP^{15}$ are all maximal linear subspaces of $V(\perm_4'')$. But their common intersection is empty, which contradicts $V(\perm_4'')$ being a cone. We conclude that $\pr(\perm_4'')>3$, which in turn implies $\pr(\perm_4')>4$, which finally implies $\pr(\perm_4)>5$.  

\end{ex}
\subsection*{Acknowledgements} We thank Zach Teitler and the anonymous referees for helpful comments. The first author was partially supported by an NSERC Discovery Grant.

\bibliographystyle{amsalpha}
\renewcommand{\MR}[1]{\relax}
\bibliography{fano-split}
\end{document}